\documentclass[10pt, a4paper]{amsart}
\usepackage[dvips]{epsfig}
\usepackage{amsmath}
\usepackage{amssymb}
\usepackage{amsfonts}
\usepackage{amsthm}
\usepackage{amsbsy}
\usepackage{amsgen}
\usepackage{amscd}
\usepackage{amsopn}
\usepackage{amstext}
\usepackage{amsxtra}
\usepackage{xypic}

\newtheorem{theorem}{Theorem}[section]

\newtheorem{lemma}[theorem]{Lemma}

\theoremstyle{definition}

\numberwithin{equation}{section}

\begin{document}

\title[A Quillen model category structure on some categories of comonoids]
{A Quillen model category structure on some categories of comonoids}
\author{Alexandru E. Stanculescu}
\address{Department of Mathematics and Statistics, McGill University, 805 Sherbrooke Str. West, Montr\'eal,
Qu\'ebec, Canada, H3A 2K6} \email{alemst@gmail.com}


\begin{abstract}
We prove that for certain monoidal (Quillen) model categories,
the category of comonoids therein also admits a model structure.
\end{abstract}

\maketitle

\section{Introduction}
A \emph{monoidal model category} is a closed symmetric monoidal category which
admits a Quillen model category structure compatible in a certain sense with the monoidal
product \cite{Ho},\cite{SS}. The majority of the natural occurring examples of model categories are monoidal
model categories. In \cite{SS}, the authors gave a sufficient condition which ensured that the category of
monoids in a monoidal model category admits a model structure, extended in an appropriate sense from the base
category. This condition was called \emph{the monoid axiom}, and it is satisfied in many examples.

Dually, one can consider comonoids in a monoidal category which has a model structure
and ask for a model structure for comonoids, somehow inherited from the base category. We were not
able to find in the literature a general result along these lines. The situation turns out to be
more complicated than with monoids. In this note we give a (very) partial answer to this problem. We prove
\begin{theorem}
Let $\mathcal{E}$ be a symmetric monoidal category with unit $I$ and let $Comon(\mathcal{E})$ be the category
of (coassociative and counital) comonoids in $\mathcal{E}$. We assume that

$(i)$ $\mathcal{E}$ is locally presentable, abelian and the monoidal product
preserves colimits and finite limits in each variable;

$(ii)$ $\mathcal{E}$ has two classes of maps $\mathrm{W}$ and $\mathrm{Cof}$ such that
$\mathrm{Cof}$ and the class of monomorphisms of $\mathcal{E}$ are the cofibrations of two model
structures on $\mathcal{E}$ with the same class $\mathrm{W}$ of weak equivalences;
furthermore, either of the two model structures is cofibrantly generated;

$(iii)$ the pushout-product axiom between the two model structures holds:
if $i:K\rightarrow L$ belongs to $\mathrm{Cof}$ and $i':X\rightarrow Y$ is a monomorphism,
then the canonical map $$K \otimes Y\underset{K\otimes X}\bigcup L\otimes X\longrightarrow L\otimes Y$$
is a monomorphism, which is a weak equivalence if either one of $i$ or $i'$ is;

$(iv)$ $I$ is $\mathrm{Cof}$-cofibrant and $\mathcal{E}$ has a coalgebra interval,
by which we mean a factorisation of the codiagonal
\[
   \xymatrix{
I\sqcup I \ar[rr]^{\bigtriangledown} \ar[dr]_{i_{0}\sqcup i_{1}} & & I\\
& Cyl(I) \ar[ur]_{p}\\
  }
  \]
such that $i_{0}\sqcup i_{1}$ belongs to $\mathrm{Cof}$, $p$ is a weak equivalence and the whole
diagram lives in $Comon(\mathcal{E})$.

Then $Comon(\mathcal{E})$ admits a model category structure in which a map is
a weak equivalence (resp. cofibration) if and only if the underlying map is a weak
equivalence (resp. monomorphism) in $\mathcal{E}$.
\end{theorem}
An analogue of 1.1 for the category of comodules over a comonoid
in $\mathcal{E}$ is presented in section 3.

One of the motivations for writing this note was the paper \cite{GG}. In \cite{AC}, the authors
extended the main result of \cite{GG} to the category of cooperads, or $F_2$-comonoids,
in the category of non-negatively graded chain complexes of vector spaces. We do not know whether
the technique used in this paper would provide a model structure on the category of cooperads
in $\mathcal{E}$.

\section{Proof of theorem 1.1}
In order to prove theorem 1.1 we shall use two results of J.H. Smith, recalled below.
\begin{theorem}(\cite{Bek}, Thm. 1.7) Let $\mathcal{E}$ be a locally
presentable category, $\mathrm{W}$ a full accessible subcategory of
$\mathrm{Mor}(\mathcal{E})$, and $\mathbf{I}$ a set of morphisms of
$\mathcal{E}$. Suppose they satisfy:

$c0$: $\mathrm{W}$ has the three-for-two property.

$c1$: $\mathrm{inj}(\mathbf{I})\subseteq \mathrm{W}$.

$c2$: The class $\mathrm{cof}(\mathbf{I})\cap \mathrm{W}$ is closed
under transfinite composition and under pushout.

Then setting weak equivalences:=$\mathrm{W}$,
cofibrations:=$\mathrm{cof}(\mathbf{I})$ and
fibrations:=$\mathrm{inj}(\mathrm{cof}(\mathbf{I})\cap \mathrm{W}$),
one obtains a cofibrantly generated model structure on
$\mathcal{E}$.
\end{theorem}
\begin{theorem} The class of weak equivalences of a combinatorial model category is accessible.
\end{theorem}
Proofs of the preceding theorem has been given in \cite{Lu} and \cite{Ros}. By general
arguments the forgetful functor $U:Comon(\mathcal{E})\rightarrow \mathcal{E}$ has a right
adjoint and the category $Comon(\mathcal{E})$ is locally presentable, see e.g.
(\cite{AR}, Remark below Lemma 2.76 and the dual of Corollary 2.75). We shall define a set
$\mathbf{I}$ which will generate the class of cofibrations and then check condition c1 of 2.1.

Let $C\in Comon(\mathcal{E})$. We say that $(D,i)\in Comon(\mathcal{E})/C$ is an
$\mathcal{E}$-\emph{subobject} of $C$ if $U(i):U(D)\rightarrow U(C)$ is a monomorphism.
As pointed out to us by Steve Lack, the $\mathcal{E}$-subobjects are precisely the strong subobjects in
$Comon(\mathcal{E})$. This can be seen using the left exactness of the monoidal product.

For example, if $f:C\rightarrow D$ is a map of comonoids, then the subobject
$m:Im(f)\rightarrowtail U(D)$ is an $\mathcal{E}$-subobject of $D$ and the canonical
epi $e:U(C)\rightarrow Im(f)$ is a map of comonoids. To see this, one uses the fact that
$Coim(f)\cong Im(f)$ and again the left exactness of the monoidal product. For $C$ and $D$
comonoids we write $C\preceq D$ if $C$ is an $\mathcal{E}$-subobject of $D$.

\begin{lemma}
There is a regular cardinal $\kappa$ such that every
comonoid $C$ is a $\kappa$-filtered colimit $C=\mathrm{colim}C_{i}$,
with $C_{i}\preceq C$.
\end{lemma}
\begin{proof}
The functor $U$ preserves and reflects epimorphisms. Let $\lambda$ be a regular cardinal
such that $Comon(\mathcal{E})$ is locally $\lambda$-presentable and let $C$ be a comonoid.
Write $C=\mathrm{colim}D_{i}$, with canonical  arrows $\varphi_{i}:D_{i}\rightarrow C$
and with $D_{i}$ $\lambda$-presentable. Factor $U(\varphi_{i})$ as
$U(D_{i})\overset{e_{i}}\rightarrow C_{i}\overset{m_{i}}
\rightarrow U(C)$, with $m_{i}$ mono and $e_{i}$ epi. By the above, $C_{i}\preceq D$ and one clearly
has $C=\mathrm{colim}C_{i}$. Since $Comon(\mathcal{E})$ is co-well-powered, there is a \emph{set}
(up to isomorphism) $Q$ of all quotients of all $\lambda$-presentable objects. Therefore there is a
regular cardinal $\kappa$ such that $Q$ is contained in the set of all $\kappa$-presentable
objects of $Comon(\mathcal{E})$.
\end{proof}
We define $\mathbf{I}$ to be the set of all isomorphism classes of cofibrations $A\rightarrow B$
with $B$ $\kappa$-presentable.
\begin{lemma} A map has the right lifting property with respect to
the cofibrations iff it has the right lifting property with respect
to the maps in $\mathbf{I}$.
\end{lemma}
To prove this lemma we need the following general result.
\begin{lemma}
Let $\mathcal{E}$ be an abelian and monoidal category with monoidal product $\otimes$
which is left exact in each variable. If $A\rightarrowtail X$ and $B\rightarrowtail Y$
are subobjects, then $A\otimes B=(A\otimes Y)\cap (X\otimes B)$. As a consequence,
if $i:D\rightarrow C$ and $j:E\rightarrow C$ are maps of comonoids in $\mathcal{E}$
such that $U(i)$ and $U(j)$ are monomorphisms, then $D\cap E$ is a comonoid in $\mathcal{E}$.
\end{lemma}
\begin{proof}
For the first part, start with the short exact sequences
$0\rightarrow A\rightarrow X\rightarrow X/A\rightarrow 0$ and
$0\rightarrow B\rightarrow Y\rightarrow Y/B\rightarrow 0$. By tensoring them one
produces a $3\times 3$ diagram all whose rows and columns are exact. The assertion follows from the
nine-lemma. For the second part, consider the cube diagram in $\mathcal{E}$
\[
   \xymatrix{
D\cap E \ar[rr] \ar[dr] \ar[dd] & &
E \ar[dr]^{j} \ar[dd]\\
&  D \ar[rr]|-{i} \ar[dd] & & C \ar[dd]\\
P \ar[rr] \ar[dr] &  &
E\otimes E \ar[dr]^{j\otimes j}\\
&  D\otimes D \ar[rr]_{i\otimes i} & & C\otimes C\\
 }
   \]
in which the top and bottom faces are pullbacks. The bottom face can be calculated as an
iterated pullback
\[
   \xymatrix{
P \ar[r] \ar[d] & (D\cap E)\otimes E \ar[d] \ar[r] & E\otimes E \ar[d]^{j\otimes E}\\
D\otimes (D\cap E) \ar[r] \ar[d] & D\otimes E \ar[r] \ar[d] & C\otimes E \ar[d]^{C\otimes j}\\
D\otimes D \ar[r]_{D\otimes i} & D\otimes C \ar[r]_{i\otimes C} & C\otimes C,\\
}
   \]
therefore $P$ is $(D\cap E) \otimes (D\cap E)$ by the first part. This provides
$D\cap E$ a comultiplication. The counit of $D\cap E$ is the counit of $C$ restricted to $D\cap E$.
\end{proof}

\begin{proof} (of lemma 2.4) The proof is standard. Let
\[
   \xymatrix{
  C \ar[r]^{f} \ar[d]_{i} & X \ar[d]^{p}\\
  D \ar[r] & Y\\
  }
  \] be a commutative diagram with $i$ a cofibration and $p$ having
the right lifting property with respect to the maps in $\mathbf{I}$.
Let $S$ be the set consisting of pairs $(E,l)$, where $C\preceq E\preceq
D$ and $l:E\rightarrow X$ is a morphism making the diagram
\[
   \xymatrix{
  C \ar[rr] \ar[d] & & X \ar[d]^{p}\\
  E \ar[r] \ar[urr]^{l} & D \ar[r] & Y\\
  }
  \] commutative. We order $S$ by $(E,l)\leqslant (E',l')$ iff
$E\preceq E'$ and $l'$ is an extension of $l$. Then $S$ is
nonempty, as it contains $(C,f)$. Let $\mathcal{C}$ be any chain in $S$ and let $\kappa'$ be a regular
cardinal such that both $\mathcal{E}$ and $Comon(\mathcal{E})$ are locally $\kappa'$-presentable.
Then $\mathcal{C}$ is $\kappa'$-directed and therefore $\textrm{colim} \mathcal{C}$ is defined in $Comon(\mathcal{E})$,
and $U(\textrm{colim} \mathcal{C})$ is the colimit of the $U(F)$, $(F,m)\in \mathcal{C}$. Hence
$\textrm{colim} \mathcal{C}\rightarrow D$ is a cofibration. Also, we have a unique
$l:\textrm{colim} \mathcal{C}\rightarrow X$ extending each $m$, and clearly
$(\textrm{colim} \mathcal{C},l)$ is an element of $S$. This shows that Zorn's lemma is applicable,
therefore the set $S$ has a maximal element $(E,l)$. We are going to show that $E\cong D$ by showing
that for each $\kappa$-presentable comonoid $B\preceq D$, one has
$B\preceq E$. This suffices since $D$, being the $\kappa$-filtered
colimit of all of its $\mathcal{E}$-subobjects, is the least upper bound of its
$\kappa$-presentable $\mathcal{E}$-subobjects.

Take $B\preceq D$ with $B$ $\kappa$-presentable.
Using lemma 2.5 and the hypothesis we have a diagonal filler $d$ in the commutative diagram
\[
   \xymatrix{
E\cap B \ar[r] \ar[d] & E \ar[r]^{l} & X \ar[d]^{p}\\
B \ar[r] & D \ar[r] & Y.\\
  }
  \]
Therefore in the diagram
\[
   \xymatrix{
E\cap B \ar[r] \ar[d] & B \ar[d] \ar[ddr]^{d}\\
E \ar[r] \ar[drr]_{l} & E\cup B \ar[dr]^{l'}\\
& & X\\
  }
  \]
in which the square is a pushout, there is a map $l':E\cup B\rightarrow X$ extending $l$,
and so $(E\cup B,l')\in S$. This shows that $(E\cup B,l')\leqslant (E,l)$ since $(E,l)$ was maximal.
It follows that $B\preceq E$. \end{proof}

By performing the small object argument it follows from
lemma 2.4 and a retract argument that the class of cofibrations is the
class $Cof(\mathbf{I})$. It remains to check condition c1 of 2.1. For this we shall use
\\

\textbf{2.6. The dual of Quillen path-object argument.} \emph{Let
$\mathcal{E}$ be a model category and let
$$F:\mathcal{C}\rightleftarrows \mathcal{E}:G$$ be an adjoint pair ($F:\mathcal{C}\rightarrow \mathcal{E}$
is the left adjoint). We define a map $f$ of $\mathcal{C}$ to be a} $\textrm{cofibration}$
(\emph{resp.} $\textrm{weak equivalence}$) \emph{if $F(f)$ is such
in $\mathcal{E}$. Suppose that $\mathcal{C}$ is finitely cocomplete,
it has a cofibrant replacement functor and a functorial cylinder
object for cofibrant objects. Then a map of $\mathcal{C}$ that has
the right lifting property with respect to all cofibrations is a
weak equivalence.}

\begin{proof} We recall its proof for the sake of completeness. Let
$f:X\rightarrow Y$ be map of $\mathcal{C}$ which has the right
lifting property with respect to all cofibrations. Let
\[
   \xymatrix{
  \hat{C}X \ar[r]^{i_{X}} \ar[d]_{\hat{C}(f)} & X \ar[d]^{f}\\
  \hat{C}Y \ar[r]^{i_{Y}} & Y\\
  }
  \]
be the cofibrant replacement of $f$. Then the diagram
\[
   \xymatrix{
  \emptyset \ar[r] \ar[d] & \hat{C}X \ar[r]^{i_{X}} & X \ar[d]^{f}\\
  \hat{C}Y \ar@{=}[r] & \hat{C}Y \ar[r]^{i_{Y}} & Y\\
  }
  \]
has a diagonal filler $d$. Let $\hat{C}X \sqcup \hat{C}X
\overset{i_{0}\sqcup i_{1}}\longrightarrow Cyl(\hat{C}X)
\overset{p}\rightarrow \hat{C}X$ be the cylinder object for
$\hat{C}X$. Consider the commutative diagram
\[
   \xymatrix{
  \hat{C}X \sqcup \hat{C}X \ar[rr]^{(d\hat{C}(f),i_{X})}
   \ar[d]_{i_{0}\sqcup i_{1}} & & X \ar[d]^{f}\\
  Cyl(\hat{C}X) \ar[rr]^{fi_{X}p} & & Y.\\
  }
  \]
By hypothesis it has a diagonal filler $H$, and so $d\hat{C}(f)$ is
a weak equivalence. Since the weak equivalences of $\mathcal{E}$
satisfy the two out of six property, it follows that $d$ is a weak
equivalence.
\end{proof}
We return to the proof of 1.1. By 2.6 it suffices to show that there is
a functorial cylinder object for comonoids. This is guaranteed by hypotheses
($iii$) and ($iv$). The proof of theorem 1.1 is complete.
\\

\textbf{Remark 2.7.}  Let $\mathcal{E}$ be as in the statement of theorem
1.1. If, moreover, the cylinder object $Cyl(I)$ for $I$ is a cocommutative comonoid,
then the category $CComon(\mathcal{E})$ of cocommutative comonoids in $\mathcal{E}$ admits a
model category structure in which a map is a weak equivalence (resp.
cofibration) if and only if the underlying map is a weak equivalence (resp.
monomorphism) in $\mathcal{E}$.
\\

\textbf{Examples.} $(a)$ Let $R$ be a commutative von Neumann regular ring and let
$Ch(R)$ be the category of unbounded chain complexes of $R$-modules.
We consider on $Ch(R)$ the projective and injective model structures
\cite{Ho}. $Ch(R)$ has a well-known coalgebra interval given by
$$...\rightarrow 0\rightarrow Re\overset{\partial}\rightarrow
Ra\oplus Rb \rightarrow 0\rightarrow ...,$$ where $\partial(e)=b-a$
and $Ra\oplus Rb$ is in degree 0. The maps $i_{0}$ and $i_{1}$ are
the inclusions and the map $p$ is $a,b \mapsto 1$, see e.g.
(\cite{SS}, section 5). The last part of ($i$) is shown in
(\cite{Sh}, Proof of Prop. 3.3 for $Ch$).

$(b)$ The above considerations apply to the category of non-negatively
graded chain complexes as well.

\section{Comodules}

Let $\mathcal{E}$ be a monoidal category with monoidal product $\otimes$. Given a
(coassociative and counital) comonoid $C$ in $\mathcal{E}$, we denote by $Mod^{C}$
the category of right $C$-comodules in $\mathcal{E}$. There is a forgetful-cofree adjunction
$$U:Mod^{C}\rightleftarrows \mathcal{E}:-\otimes C \qquad \qquad \qquad (1)$$

\begin{theorem}
Let $\mathcal{E}$ be a cofibrantly generated monoidal model category and let
$C$ be a (coassociative and counital) comonoid in $\mathcal{E}$. Suppose that

$(i)$ the cofibrations of the model structure are precisely the monomorphisms;

$(ii)$ $\mathcal{E}$ is locally presentable, abelian, and for each object $X$ of $\mathcal{E}$
the functor $-\otimes X$ is left exact, where $\otimes$ denotes the monoidal product
of $\mathcal{E}$.

Then $Mod^{C}$ admits a cofibrantly generated model structure in which a map $f$
is a weak equivalence (resp. cofibration) if and only if the underlying map
is a weak equivalence (resp. monomorphism) in $\mathcal{E}$.
\end{theorem}
The proof of the above theorem follows the same steps as the proof of 1.1,
except that condition c1 of 2.1 will be a consequence of lemma 3.2 below.
\\

We say that a map of $C$-comodules is a {\bf fibration} if
it has the right lifting property with respect to the maps which are
both cofibrations and weak equivalences. We say that a map of $C$-comodules
is a  {\bf trivial fibration} if it is both a fibration and a weak equivalence.
\begin{lemma}
The category $Mod^{C}$ has a weak factorisation system (cofibrations, trivial fibrations).
\end{lemma}
\begin{proof}
We follow an idea from \cite{GG}. Let $f:M\rightarrow N$ be a map of $C$-comodules.
We factor the map $U(M)\rightarrow 0$ as a monomorphism followed by a
trivial fibration $U(M)\overset{i}\rightarrow X\rightarrow 0$.
Then $f$ factors as $$M\overset{j}\rightarrow
N\times (X\otimes C) \overset{p_{1}}\rightarrow N$$ where $j=(i^{*},f)$,
$i^{*}$ is the adjoint transpose of $i$ and $p_{1}:N\times (X\otimes C)
\rightarrow N$ is the projection. The map $p_{1}$ is a weak equivalence since
it is the map $N\oplus (X\otimes C)\rightarrow N\oplus (0\otimes C)\cong N$,
which is a weak equivalence. We show that the underlying map of $j$
is a monomorphism. One has $i=\epsilon_{X}U(p_{2}j)$, where
$\epsilon_{X}$ is the counit of the adjunction (1) and
$p_{2}:N\times (X\otimes C) \rightarrow (X\otimes C)$ is the projection.
Therefore $j$ is a cofibration. Next we show that $p_{1}:N\times (X\otimes C)
\rightarrow N$ has the right lifting property with respect to all cofibrations.
Let
\[
   \xymatrix{
M' \ar[r] \ar[d]_{k}  & N\times (X\otimes C) \ar[d]^{p_{1}}\\
N' \ar[r] & N\\
}
   \]
be a commutative diagram with $k$ a cofibration. This diagram has a
diagonal filler if and only if the diagram
\[
   \xymatrix{
  U(M') \ar[r] \ar[d]_{U(k)} & X \ar[d]\\
  U(N') \ar[r] & 0\\
}
   \]
has one. The latter is true by the assumption on $X$.
Therefore $p_{1}$ is a fibration. Let now $f:M\rightarrow N$ be a
trivial fibration. Factor it as above $M\overset{j}\rightarrow
N\times (X\otimes C) \overset{p_{1}}\rightarrow N$. Since
$j$ is a weak equivalence, there is a diagonal filler in the diagram
\[
   \xymatrix{
  M \ar @{=} [r] \ar[d]_{j} & M \ar[d]^{f}\\
 N\times (X\otimes C) \ar[r] \ar @{..>} [ur] & N\\
}
   \]
hence $f$ is a (domain) retract of a map which has
the right lifting property with respect to all cofibrations, therefore
$f$ has the right lifting property with respect to all cofibrations.
Conversely, let $f:M\rightarrow N$ have the right lifting property
with respect to all cofibrations. The same argument shows that $f$ is a
(domain) retract of a trivial fibration, hence $f$ is a trivial fibration.
\end{proof}

\textbf{Acknowledgements.} We are indebted to Professor Michael
Makkai for many useful discussions and suggestions. We thank the referee
for useful comments.

\end{document}